\documentclass[11pt]{article}
\usepackage[utf8]{inputenc}
\usepackage[T1]{fontenc}
\usepackage{amsfonts}
\usepackage{amsmath}
\usepackage{color}
\usepackage[left=2.6cm,right=2.6cm,top=2.5cm,bottom=2.5cm]{geometry}
\usepackage{graphicx}
\usepackage{natbib}

\newtheorem{theorem}{Theorem}

\newtheorem{example}[theorem]{Example}

\newtheorem{lemma}[theorem]{Lemma}

\newtheorem{proposition}[theorem]{Proposition}
\newtheorem{remark}[theorem]{Remark}

\newenvironment{proof}[1][Proof]{\noindent \textbf{#1.} }{\  \rule{0.5em}{0.5em}}

\begin{document}

\title{Spatial competition with unit-demand functions\thanks{%
We thank Helmuth Cremer for helpful discussions about earlier drafts of this
paper. Van der Straeten acknowledges funding from the French National
Research Agency (ANR) under the Investments for the Future (Investissements
d'Avenir) program, grant ANR-17-EURE-0010.}}
\author{Ga\"{e}tan\textsc{\ Fournier}\thanks{%
Aix-Marseille Univ., CNRS, EHESS, Centrale Marseille, AMSE: gaetan.fournier@univ-amu.fr}\textsc{, }Karine 
\textsc{Van der Straeten}\thanks{%
CNRS, Toulouse School of Economics \& Institute for Advanced Study in
Toulouse: karine.van-der-straeten@tse-fr.eu}, \& J\"{o}rgen W. \textsc{Weibull}\thanks{%
Stockholm School of Economics: jorgen.weibull@hhs.se} }
\maketitle

\begin{abstract}
This paper studies a spatial competition game between two firms that sell a
homogeneous good at some pre-determined fixed price. A population of
consumers is spread out over the real line, and the two firms simultaneously
choose location in this same space. When buying from one of the
firms, consumers incur the fixed price plus some transportation
costs, which are increasing with their distance to the firm. Under the
assumption that each consumer is ready to buy one unit of the good whatever
the locations of the firms, firms converge to the median location: there is
\textquotedblleft minimal differentiation\textquotedblright . In this
article, we relax this assumption and assume that there is an upper
limit to the distance a consumer is ready to cover to buy the good. We show
that the game always has at least one Nash equilibrium in pure strategy.
Under this more general assumption, the \textquotedblleft minimal
differentiation principle\textquotedblright ~no longer holds in general. At
equilibrium, firms choose \textquotedblleft minimal\textquotedblright ,
\textquotedblleft intermediate\textquotedblright ~or \textquotedblleft
full\textquotedblright ~differentiation, depending on this critical distance
a consumer is ready to cover and on the shape of the distribution of
consumers' locations.
\end{abstract}


\section{Introduction}
The choice of product characteristics - and strategic product
differentiation in particular - is a central issue in Industrial
Organization.

A large number of studies on this topic build on Hotelling's\ seminal model
of firm location (\cite{hotelling1929stability}). In Hotelling's model,
consumers are uniformly distributed on a line. Two firms selling an
homogeneous good simultaneously and non-cooperatively choose a location on
this line (stage 1). Once locations are observed, firms simultaneously
choose a price at which they sell the good (stage 2). Consumers are ready to
buy exactly one unit of the good (whatever the prices and the locations).
They incur linear transportation costs when traveling on the line to
purchase the good. Hotelling claims that this two-stage game has a unique
Nash equilibrium, where both firms choose the same location at stage 1 -
hence the name of \textquotedblleft minimal differentiation
principle\textquotedblright \ given to Hotelling's result.

Note that the game is framed here in geographical terms, but there is an
immediate analogy with a situation where firms, instead of a geographical
location, choose some characteristics of their products in some space
product \`{a} la Lancaster, and consumers differ in their preferences for
product characteristics. In this interpretation, the counter-part of the
transportation cost is the utility loss suffered by a consumer who consumes
a product whose characteristics do not exactly match her preferred ones. In
the paper, we will use the geography terminology, and talk about firms'
\textquotedblleft positions\textquotedblright ~or \textquotedblleft locations\textquotedblright, but all the results can also be equally
interpreted in terms of more general product characteristics. \newline

\cite{d1979hotelling} challenge Hotelling's convergence result,
demonstrating that there is a flaw in the resolution of the price subgame
stage: the price subgame has no equilibrium in pure strategies when firms
are too close one from this other. Assuming that consumers have quadratic
transportation costs (instead of linear as in the original article), they
show that the price subgame always has a pure strategy equilibrium. They
show that in that case, a \textquotedblleft maximum
differentiation\textquotedblright \ principle holds, firms locating at the
two ends of the line. The powerful intuition behind this result is that
firms differentiate to avoid too fierce a price competition at the second
stage.\footnote{%
See also \cite{economides1986minimal}, \cite{osborne1987equilibrium} and 
\cite{bester1996noncooperative} for further discussion on Hotelling's result.%
} As summarized by \cite{tirole1988theory}, this two-stage model where firms
first choose product characteristics and then choose prices offers the
standard explanation in industrial organization as to why \textquotedblleft(...) \textit{%
firms generally do not want to locate at the same place in the product
space. The reason is simply the Bertrand paradox: Two firms producing
perfect substitutes face unbridled price competition (at least in a static
framework). In contrast, product differentiation establishes clienteles
(\textquotedblleft market niches\textquotedblright , in the business
terminology) and allows firms to enjoy some market power over these
clienteles. Thus, firms usually wish to differentiate themselves from other
firms}\textquotedblright \ (\cite{tirole1988theory}, page 278).\footnote{%
A few papers have also analyzed how uncertainty about demand impacts this
incentive to differente, such as \cite{de1985principle} and \cite%
{meagher2004product}.}

\bigskip

In the present paper, Hotelling's convergence result is also challenged, but
on completely different grounds. We argue that softening the price
competition is not the only force which may drive firms apart. We do so by
relaxing Hotelling's assumption that the market is always covered, whatever
the locations and prices of the firms. Instead of assuming perfectly
inelastic demand, we assume unit demand funtions: a consumer buys the good
only if her valuation for the good is higher than the total cost, where the
total cost is the price of the good augmented by the transporation cost. If
for both firms, the total cost is lower than her valuation, the consumer
buys from the firm with the lowest total cost (and randomizes equally
between the two firms in case of equality). Under this more general
assumption, if both firms are too far away from her location, a consumer
might prefer not to buy the good. We will show that introducing this option
to abstain/stay out of the market, can be a powerful force in favor of
differentiation.

In order to make this argument as transparent as possible, we study a
one-stage location game, where firms are assumed to sell the good at some
pre-determined fixed price. This will allow us to clearly distinguish our
effect from the one driven by the price competition.

\bigskip

Note that this fixed-price situation is interesting \textit{per se}, since
there are many situations where, for legal or technical reasons, price is
not a free parameter in the competition. As noted by \cite{tirole1988theory}%
, \textquotedblleft \textit{There may exist legal or technical reasons why
the scope of price competition is limited. For instance, the prices of
airline tickets in the United States (before deregulation) where determined
exogenously, as the price of gas and books in France once were}%
\textquotedblright \ (page 287). Some shops sell products whose price is
exogenously determined, for instance newsstands, pharmacies, or franchises
of brand clothes for example.\footnote{%
In the survey by Gabszewicz and Thisse (1992) of the early literature on
spatial competition, Section 4 (pages 298-302) is devoted to this
fixed-price situation.}

\bigskip

In this fixed-price spatial competition, if the demand is assumed to be
perfectly inelastic, both firms choose the \textquotedblleft median
consumer\textquotedblright \ location, i.e. the location such that one half
of the consumers lay on its left-hand side, and the other half lay on its
right-hand side. This results holds whatever the form of the transportation
costs and the distribution of consumers.\footnote{%
This result has been known in political economy under the name of
\textquotedblleft Median Voter theorem\textquotedblright (see \cite%
{black1948rationale}, \cite{downs1957economic}). In this interpretation, two
political parties compete to attract voters, who are located along an
\textquotedblleft ideological\textquotedblright left-right axis. If parties
seek to attract as many voters as possible, they will at equilibrium both
choose the median location. Note that this analogy was already noted by
Hotelling in his seminal article, where he writes: \textquotedblleft \textit{%
So general is this [agglomerative] tendency that it appears in the most
diverse fields of competitive activity, even quite apart from what is called
economic life. In politics it is strikingly exemplified. The competition for
votes between the Republican and Democratic parties does not lead to a clear
drawing of issues, an adoption of two strongly contrasted positions between
which the voter may choose. Instead, each party strives to make its platform
as much like the other's as possible.}\textquotedblright (page 54)} The
intuition behind this convergence result is quite powerful. Consider any
situation where the firms choose different locations. Then both firms could
increase their profit by moving closer to their opponent. Indeed, with such
a move, each firm would win additional consumers (among those initially
located between the two firms), without losing any consumers on the other
side. This shows that at any equilibrium in pure strategies, the firms
should converge.\footnote{%
Note that the result crucially depends on the duopoly assumption: with more
than two firms, the situation in which all firms converge at the median is
no longer an equilibrium. Indeed, consider a pure location game with $n\geq
2 $ firms. If all firms locate at the median, each gets a share $1/n$ \ of
the consumers. By moving slightly to left or the right, a deviating firm
could attract almost $1/2$ of the consumers. This simple argument shows that
convergence of all firms at the median cannot be an equilibrium at soon at $%
n\geq 3$. In such a game, it has been shown that an equilibrium in pure
strategy exists only under very restricted assumptions about the
distribution of consumers, and when it does, firsm do not all converge, see
for example Eaton and Lipsey [1975] and \cite{fournier2016general} . Peters
et al. [2018] study a pure location model with congestion where consumers
are uniformly distributed and each consumer selects one of the firms based
on distances as well as the number of consumers visiting each firm. They
provide conditions for the existence of subgame perfect Nash equilibrium,
and show that firms do not converge when the number of firms is larger than
two.}\newline

If instead there exists a maximum distance that consumers are ready to
travel to buy the good, we show that the convergence result may not hold
anymore. One may observe some \textquotedblleft
intermediate\textquotedblright ~or even \textquotedblleft full
differentiation\textquotedblright . To be more precise about what we mean by
partial or full differentiation, we define the \textquotedblleft potential
attraction zone\textquotedblright ~of a firm as the set of consumers who
prefer buying from this firm rather than not buying the good at all. We say
that there is \textquotedblleft partial differentiation\textquotedblright
~when firms choose different locations but their potential attraction zones\
intersect; and that there is \textquotedblleft full
differentiation\textquotedblright ~when the two potential attraction zones do
not intersect (or intersect over of set of consumers of measure $0$). We
characterize all pure strategy equilibria, and discuss their properties
under quite general assumptions about the transportation cost functions and
the distribution of consumers. Assuming mild assumptions on the distribution
of consumers\footnote{%
We assume that the distribution is single-peaked, symmetric and that it has
a continuous log-concave density. Most standard distributions satisfy these
assumptions.}, our results are the following. If the maximum traveling
distance is high enough, both firms converge to the median/modal location%
\footnote{%
Since we assume that $f$ is single-peaked and symmetric, the median and
modal location coincide.} (the standard convergence result). Now, if this
distance is small enough, firms diverge at equilibrium. To understand the
main intuition behind this result, suppose that a firm has chosen the
median/modal position. In that case, if its opponent also selects this
central position, the two firms will have exactly the same potential
attraction zones, and thus each will attract one half of the consumers who
are located within acceptable distance of the central position. The latter
firm may fare better in that case by avoiding this frontal competition, and
moving somewhat to the left or the right. In doing so, it might win new
consumers located \textquotedblleft at the periphery\textquotedblright, although it will come at the cost of
losing some \textquotedblleft central\textquotedblright ~consumers. We expect the incentives to move away to be
greater when the distribution is flatter (less concentration at the modal
position) and when the width of the attraction zone is larger. It will be
shown to be indeed the case. Depending of the width of the attraction zone
compared to some indicator of the flatness of the distribution of consumers'
location, we can observe full convergence to the central position,
intermediate differentiation, or complete differentiation (in the sense that
no consumer is located at equilibrium within acceptable distance of both
firms). In particular, some necessary and sufficient conditions on this
ratio are provided for the convergence result to hold.

\bigskip

We are not the first to revisit Hotelling's assumption of perfectly
inelastic demand. Early contributions by \cite{lerner1937some} and \cite%
{smithies1941optimum} note the centrifugal forces that a more elastic demand
may generate. \cite{economides1984principle} study a two-stage
location-then-price Hotelling game where consumers have a finite valuation
for the good. In that case, even with linear costs of transportation, a
price equilibrium may exist at the second stage, and in the first stage,
firms may differentiate.\footnote{%
However, because a price equilibrium does not exist for every pair of
locations, a complete analysis of the spatial competition is impossible. The
paper focuses on local firms' deviations.} Imperfectly inelastic demand in a
pure location game has been studied by \cite{feldman2016variations} and \cite%
{shen2017hotelling}, who consider a model where each seller has an interval
of attraction, as it is the case in our model, but they suppose that
consumers randomly select where to buy among attractive sellers. Contrary to
our assumption, buyers do not necessarily buy from the closest place. \cite%
{feldman2016variations} study the case of uniformly distributed consumers;
whereas \cite{shen2017hotelling} study more general distributions. They
prove the existence of pure Nash equilibrium, but do not describe it. In the
political science literature (see Footnote 4 for the analogy between a
fixed-price location game between firms and an electoral competition game
between parties), \cite{downs1957economic}, \cite{hinich1970plurality} or
more recently \cite{xefteris2017simple} have also noted that if voters
prefer to abstain when neither party is close enough to their ideal policy,
differentiation may result at equilibrium. \cite{xefteris2017simple} study a
more general abstention function, and show that the game admits an
equilibrium in mixed strategies (existence result). They characterize
equilibria in pure strategies only under the assumption that voters are
uniformly distributed and for a special case of the abstention function. 
\cite{hinich1970plurality} mostly focus on the comparison of parties'
objectives: plurality maximization versus vote maximization. In the latter case,
which is the one we study in this paper, they show that differentiation can
occur in equilibria. We provide a more complete characterization of all
equilibria.\footnote{%
In particular, because in their paper they mostly focus on first order
conditions, they fail to notice that a continnum of asymmetric equilibra may
exist under some configurations of the parameters, and that first order
conditions are not necessarily sufficient.} 

\bigskip

The paper is organized as follows. The model is presented in more detail in
section 2. Section 3 characterizes all Nash equilibria in which firms play
pure strategies. Section 4 comments the results and discusses some of the
assumptions. Section 5 contains the proofs.

\section{The model}

We study a (fixed-price) spatial competition game between two firms facing
consumers with unit demand functions.

\begin{itemize}
\item The two firms ($i=1,2$) produce the same homogeneous good. Firm $i\in
\{1,2\}$ produces quantity $q$ at cost $\gamma _{i}(q)$. Firms sell the good
at some identical pre-determined fixed price $p>0$. Before selling the good,
they simultaneously select locations $x_{1}$ and $x_{2}$ on the real line $%
\mathbb{R}$.

\item A mass 1 of potential consumers is distributed on $X=\mathbb{R}$
according to a probability distribution that is absolutely continuous with
respect to the Lebesgue measure. We denote $f$ its density, and $F$ its
cumulative distribution. We focus our analysis on the set $\mathcal{D}$ of
distributions that have continuous log-concave densities\ (i.e. such that $f$
can be written $f(x)=e^{g(x)}$ where $g$ is a concave function), and such
that $f$ is symmetric around $0$ and strictly decreasing on $\mathbb{R}^{+}$%
. The analysis would be identical if the distribution was symmetric around a
mode different from $0$. The above hypotheses describe a very large class $%
\mathcal{D}$ of distributions that contains for example the normal
(centered) distributions, the Laplace distributions, the symmetric
exponential distributions, the logistic distributions, the symmetric gamma
distributions, the symmetric extreme value distributions, etc.\footnote{%
In Section \ref{section:uniform}, we also study the case of a uniform
distribution of consumers.}

\item All consumers have the same valuation for the good $v>0$. They also
incur transportation costs: they have a utility loss of traveling a distance 
$d\geq 0$ that is denoted $c(d)$. We suppose that $c\left( 0\right) =0$ and
that $c$ is strictly increasing and continuous. If a consumer travels a
distance $d$ to buy the good at the pre-determined price $p$, she gets the
total utility: 
\begin{equation*}
u=v-p-c(d).
\end{equation*}%
If she doesn't buy the good, her utility is normalized to $0$. Assuming that 
$v-p>0$, $u$ is positive whenever the distance $d$ is smaller than $\delta $%
, where: 
\begin{equation}
\delta :=\left \{ 
\begin{array}{l}
c^{-1}(v-p)>0\text{ if }v-p\leq \lim \limits_{d\rightarrow \infty }c(d), \\ 
+\infty \text{ \  \  \  \  \ otherwise.}%
\end{array}%
\right.  \label{Definition_delta}
\end{equation}%
Parameter $\delta >0$ denotes the maximal distance that a consumer is ready
to travel to buy the good. It is strictly increasing in the valuation of the
good ($v$) and decreasing in its price ($p$). Under these assumptions, a
consumer buys from the closest firm if her distance to the firm is smaller
than $\delta $ (randomly choosing a firm if both firms are equidistant from
her own position), and she doesn't buy otherwise.

\item We assume that firms serve all the demand they face at price $p$, and
that they maximize their profit. When a quantity $q$ of consumers buy from
firm $i$, it makes a profit equal to $p\times q-\gamma _{i}(q)$. We assume $%
\gamma _{i}^{\prime }\geq 0$, $\gamma _{i}^{\prime \prime }\geq 0$ and $%
\gamma _{i}^{\prime }(1)<p$, which imply that this profit function is
strictly increasing with respect to $q$. Under these assumptions, maximizing
its profit is equivalent for the firm to maximizing the quantity it sells.

\item We can now formally define the 2-player game $\mathcal{H}(f,\delta )$
associated to distribution $f$ and parameter $\delta $. The firms
simultaneously select locations $x_{1}$ and $x_{2}$ in $\mathbb{R}$. We
denote by $q_{i}(x_{1},x_{2})$ the quantity of consumers who buy from firm $%
i $ when players choose locations $x_{1}$ and $x_{2}\in \mathbb{R}$. Since
for a firm, maximizing its profit is equivalent to maximizing the quantity
of consumers who buy from this firm, we define the payoff of firm $i$ as
being $q_{i}(x_{1},x_{2})$. Given our assumptions about consumer behavior,
the payoffs of the players are defined by:%
\begin{equation*}
q_{i}(x_{1},x_{2}):=\left \{ 
\begin{array}{c}
\displaystyle \int_{\{t~:\text{ }\left \{ \left \vert x_{i}-t\right \vert
\leq \delta \text{ and }\left \vert x_{i}-t\right \vert =\min \left \{ \left
\vert x_{1}-t\right \vert ,\left \vert x_{2}-t\right \vert \right \} \right.
\}}f(t)dt\text{ if }x_{1}\neq x_{2}, \\ 
\\ 
\displaystyle \frac{1}{2}\int_{\{t~:\text{ }\left \vert x_{i}-t\right \vert
\leq \delta \}}f(t)dt\text{ if }x_{1}=x_{2}.%
\end{array}%
\right.
\end{equation*}
\end{itemize}

We now introduce a number of definitions that will be useful to present our
main results.

\subparagraph{\textbf{Definition \textquotedblleft Potential attraction
zones\textquotedblright~: }}

We call \textit{potential attraction zone} of a firm the set of locations
such that consumers at these locations prefer buying from this firm rather
than not buying the good at all. Formally, the potential attraction zone of
firm $i$ when locations are $(x_{1},x_{2})$, denoted by $A_{i}(x_{1},x_{2})$%
, is $A_{i}(x_{1},x_{2}):=\{t~:$ $\left \vert x_{i}-t\right \vert \leq
\delta \}$.

\subparagraph{\textbf{Definition \textquotedblleft No
differentiation\textquotedblright / \textquotedblleft Full convergence
\textquotedblright~: }}

We say that at profile of locations $(x_{1},x_{2})$, there is \emph{no} $%
\emph{differentiation}$ (or \emph{full convergence}) if the potential
attraction zones of the two firms exactly coincide. Formally, this is the
case if $A_{1}(x_{1},x_{2})=A_{2}(x_{1},x_{2})$. Note that this happens if
and only if $x_{1}=x_{2}$.

\subparagraph{\textbf{Definition \textquotedblleft Partial differentiation
\textquotedblright~: }}

We say that at profile of locations $(x_{1},x_{2})$, there is \emph{partial} 
$\emph{differentiation}$ if the potential attraction zones of the two firms
partially overlap. Formally, this means that the following two conditions
are simultaneously satisfied: (i) the two potential attraction zones $%
A_{1}(x_{1},x_{2})$ and $A_{2}(x_{1},x_{2})$ intersect over a set of
consumers of positive measure, (ii) $x_{1}\neq x_{2}$.

\subparagraph{\textbf{Definition \textquotedblleft Full
differentiation\textquotedblright~: }}

We say that at profile of locations $(x_{1},x_{2})$, there is \emph{full} $%
\emph{differentiation}$ if the potential attraction zones of the two firms
do not intersect, or intersect over a set of consumers of measure $0$.

\bigskip

Note that if $\delta =+\infty $, then for any distribution in $\mathcal{D}$
we have that $(0,0)$ is the unique equilibrium (Median Voter theorem). In
the following we focus on the case where $\delta \in ]0,+\infty \lbrack $.

In the next section, we characterize all Nash equilibria in pure strategies
of the games $\mathcal{H}(f,\delta )$ for $\delta \in ]0,+\infty \lbrack $
and $f\in \mathcal{D}$.

\section{Equilibria}

In this section, we characterize all Nash equilibria in pure strategies of
the games $\mathcal{H}(f,\delta )$. We will show that for any $\delta \in
]0,+\infty \lbrack $ and $f\in \mathcal{D}$, the game $\mathcal{H}(f,\delta )
$ always has at least one equilibrium in pure strategy.\footnote{%
We know from \cite{xefteris2017simple}(Proposition 1) that the game admits a symmetric Nash
equilibrium (possibly) in mixed strategies.} We
present these equilibria according to the level of differentiation they
entail. The main results of the paper are Proposition \ref{prop:no
differentiation}, Proposition \ref{prop:partial differentiation} and
Proposition \ref{prop:full differentiation}, which show that a pure strategy
equilibrium always exists, and characterize the set of equilibria. They deal
respectively with equilibria inducing no, partial and full differentiation.

We will show that the necessary and sufficient conditions for the existence
of these different types of equilibria only depend on parameter $\delta $
(the maximum distance a consumer is ready to travel to buy the good) and on
a parameter $\kappa $ that is defined as the positive solution of the
equation:
\begin{equation}
\frac{1}{2}f\left( 0\right) =f(\kappa ).  \label{Definition_kappa}
\end{equation}%
Note that Equation (\ref{Definition_kappa}) admits exactly one positive
solution. Indeed, $f$ is continuous, strictly decreasing on $\mathbb{R}^{+}$
with $f(0)>0$ and $\lim_{+\infty }f(x)=0$ (because $f$ is a decreasing
probability density). Parameter $\kappa $ is the time it takes for $f$ to decrease to half its
modal value.\ It measures how 'flat' the consumer distribution is.\newline

\begin{proposition}
\label{prop:no differentiation}\textbf{(No differentiation)}\newline
$(i)$ A (pure strategy) equilibrium with \emph{no differentiation} exists if
and only if $\delta \geq \kappa $. \newline
$(ii)$ In this case, the unique equilibrium of the game is $\left(
0,0\right) $: both firms converge at the median/modal position.
\end{proposition}

\begin{proposition}
\label{prop:partial differentiation}\textbf{(Partial differentiation)}%
\newline
$(i)$ A (pure strategy) equilibrium with \emph{partial differentiation}
exists if and only if $\frac{\kappa }{2}<\delta <\kappa $. \newline
$(ii)$ In this case, the unique equilibrium of the game is $\left( \delta
-\kappa ,\kappa -\delta \right) $ (up to a permutation of the players).
\end{proposition}

\begin{proposition}
\label{prop:full differentiation}\textbf{(Full differentiation)}\newline
$(i)$ A (pure strategy) equilibrium with \emph{full differentiation} exists
if and only if $\delta \leq \frac{\kappa }{2}$. \newline
$(ii)$ In this case, there is a unique symmetric equilibrium $(-\delta
,\delta )$. Besides, as soon as $\delta <\frac{1}{2}\kappa $, there is also
a continuum of asymmetric Nash equilibria, where firms are located at
distance exactly $2\delta $ one from the other. \newline
More specifically, (up to a permutation of the players) the whole set of
equilibria is $(m-\delta ,m+\delta )$ for $m\in \left[ -\alpha ,\alpha %
\right] $, where $\alpha \in \left[ 0,\delta \right] $ is uniquely defined
by:%
\begin{equation}
\alpha :=\max \{t\in \lbrack 0,\delta ]:\frac{1}{2}f\left( t\right) \leq
f(t+2\delta )\}.  \label{Definition_alpha}
\end{equation}
\end{proposition}

The proof of these three propositions is provided in the Appendix (Section %
\ref{se:proof}). Before we give in the next section an economic intuition
for these main results, note that parameters $\kappa $ and $\alpha $ are
easy to derive from the distribution $f$ of consumers, as illustrated in the
following examples.

\begin{example}
\label{Ex: Normal distribution}~~\textbf{(Normal distribution)\newline
}Suppose that consumers are distributed according to a normal distribution $%
\mathcal{N}(0,\sigma ^{2})$, i.e. $f(x)=\frac{1}{\sigma \sqrt{2\pi }}e^{-%
\frac{x^{2}}{2\sigma ^{2}}}$ with $\sigma >0$. \newline
Then: $\kappa =\sigma \sqrt{2ln(2)}$ and $\alpha =\min \left \{ \delta ,%
\frac{\sigma ^{2}ln(2)}{2\delta }-\delta \right \} $. \newline
\end{example}

\begin{example}
\label{Ex: Laplace distribution}~~\textbf{(Laplace distribution)\newline
}Suppose that consumers are distributed according to a Laplace distribution $%
\mathcal{L}(0,\beta )$, i.e. $f(x)=\frac{1}{2\beta }e^{-\frac{\left \vert
x\right \vert }{\beta }}$ with $\beta >0$. \newline
Then: $\kappa =\beta ln(2)$ and $\alpha =\delta $.
\end{example}

\section{Comments and discussion}

In this section, we first comment upon our main results, and give the main
economic intuition. We then propose some efficiency considerations. Last, we
discuss how our results should be adapted in the case of a uniform
distribution of consumers.

\subsection{Comments}

As explained in the introduction, we explain differentiation by a direct
demand-driven effect, stemming from the fact that there is a maximal
distance consumers are ready to cover to buy the good. To understand how
this effect operates, consider again the powerful argument leading to
convergence in the case of a perfectly inelastic demand ($\delta =+\infty $%
). Suppose that the firms choose different locations. Then each of them can
unambiguously increase its profit by moving closer to its competitor.
Indeed, consider the firm initially located on the right hand side, say Firm
2. By moving closer to its opponent (that is, moving to the left),

\begin{enumerate}
\item Firm 2 does not lose any consumers on its right-hand side (by
assumption, these consumers will still buy from Firm 2), and

\item Firm 2 attracts a larger quantity of consumers located between the two
firms.
\end{enumerate}

When we relax the assumption that consumers are ready to buy the good
whatever the locations of the firms, this latter argument is still active:
There is still a force towards convergence, due to the willingness to
compete for the "central" consumers. But the former argument according to
which Firm 2 does not lose any consumers on its right-hand side is no longer
valid. In that case, by moving closer to its competitor, the firm may lose
the "peripheral" consumers who were indifferent between buying from Firm 2
and not buying the good. There are now two types of relevant marginal
consumers: those who are located between the two firms and could potentially
buy from both, and those who are located at the border of the domains of
attraction and who are indifferent between buying from the closest firms and
not buying the good. Depending on how this trade-off is solved, there can be
no, partial or full differentiation at equilibrium.

Propositions \ref{prop:no differentiation}, \ref{prop:partial
differentiation} and \ref{prop:full differentiation} taken together show
that the regime regarding the firm differentiation depends on the ratio $%
\frac{\delta }{\kappa }$. Interestingly, the characterization of the
equilibria does not depend on the details of the transport cost function.
The only thing that matters is the maximal distance the consumer is ready to
cover to buy the good ($\delta $) and the shape of the distribution of
consumers, as summarized by parameter $\kappa $, where $\kappa $ is the time
it takes for the density $f$ to decrease to half its initial value (see
Equation (\ref{Definition_kappa})).

\bigskip

\textbf{Case }$\frac{\delta }{\kappa }\geq 1$\textbf{: No differentiation.}
In that case, there is a unique Nash equilibrium, at which both firms choose
to locate at the median position. As noticed in the introduction, the
intuition suggests that firms will convergence at the center if the distance
a consumer is ready to cover to buy the good is large enough ($\delta $
large) or if consumers are sufficiently numerous around the center.
Proposition \ref{prop:no differentiation} provides a precise quantification
for these conditions: The situation where both firms converge is an
equilibrium if and only if $\delta \geq \kappa $.

To understand the intuition behind this condition, assume that one firm, say
firm 1, chooses the modal median location ($0$). If its opponent also
selects this central position, both firms will have exactly the same
potential attraction zones: $A_{1}(0,0)=A_{2}(0,0)=\left[ -\delta ,\delta %
\right] ,$ and each will attract one half of the consumers who are located
within acceptable distance of $0$. Therefore, the payoff for firm 2 is $%
q_{2}(0,0)=\frac{F(\delta )-F(-\delta )}{2}$. Since $f$ is assumed to be
symmetric, note that $q_{2}(0,0)=F(\delta )-F(0)$. If firm 2 moves slightly
to the right, say by some small $\varepsilon >0$, its potential attraction
zone will now be $A_{2}(0,\varepsilon )=\left[ \delta -\varepsilon ,\delta
+\varepsilon \right] $. By doing so, it will attract all consumers located
between $\frac{\varepsilon }{2}$ and $\delta +\varepsilon $, and $%
q_{2}(0,\varepsilon )=F(\delta +\varepsilon )-F(\frac{\varepsilon }{2})$.
The move is beneficial if the mass of consumers located at $\delta $ is
larger than half the mass of the consumers located at $0$. This condition is 
$\frac{1}{2}f(0)<f(\delta )$, which is exactly the condition $\delta <\kappa 
$ (remember that $\kappa $ is the time it takes for the density $f$ to
decrease to half its initial value). The assumptions about the logconcavity
of $f$ are sufficient to guarantee that the examination of first order
conditions are sufficient to characterize equilibrium. We also show in the
appendix that there is no equilibrium with convergence at another location
than $0$.

Consider Example \ref{Ex: Normal distribution}. In the case of a normal
distribution with variance $\sigma ^{2}$, the condition $\delta \geq \kappa $
can be written as $\delta \geq \sigma \times \sqrt{2\ln (2)}$. Note that $%
\sqrt{2\ln (2)}$ is approximately equal to $1.18$. This shows that, for the
no differentiation principle to hold, the total length\ of a firm's
potential attraction zone ($2\delta $) has to be approximately at least as
large as$\ 2.35$ time the standard deviation of the distribution of consumer
locations. This figure is quite high. By instance, one may check that when $%
\delta /\sigma =\sqrt{2\ln (2)}$, the potential attraction zone of a firm
located at the center covers over 75\% of the population.

\bigskip

\textbf{Case} $\frac{1}{2}<\frac{\delta }{\kappa }<1$\textbf{: Partial
differentiation.} In that case, for each $(\delta ,\kappa )$, there is a
unique Nash equilibrium, in which the two firms engage in \textquotedblleft
partial differentiation\textquotedblright . The unique equilibrium is
symmetric, with firms choosing locations $\left( \delta -\kappa ,\kappa
-\delta \right) $, where $0<\kappa -\delta <\delta $. The distance between
the two firms is $2\left( \kappa -\delta \right) <2\delta $: a positive mass
of consumers, in particular the median consumer, are located within
acceptable distance of both firms. Note that the distance between the two
firms is decreasing in $\delta $.

To understand the intuition behind this result, remember that, as discussed
in the case of full convergence, whenever the potential attraction zones of
the two firms intersect on a set of positive mass, a firm faces a trade-off.
Indeed, by moving away from its opponent, it could attract new "peripheral"
consumers, who were not buying the good at the initial locations. But this
move would imply losing the "central" consumers who were initially
indifferent between the two firms, a share 1/2 of which were buying from
this firm in the initial situation. The equilibria described in that case
are characterized by the fact that these two effects exactly offset one
another. Besides, one can show (see Appendix) that only symmetric equilibria
exist in that case: firms choose symmetric locations, say $(-x,+x)$, $x\geq
0 $. Consider the firm at location $x$. For this firm, the "peripheral"
consumers that it could attract by moving further to the right are those
located around $x+\delta $, whereas the "central" consumers who are
initially indifferent between the two firms are those located around $0$.
The condition stating that the two opposite effects exactly counter-balance
is therefore $\frac{1}{2}f(0)=f(x+\delta )$, which yields $x=\kappa -\delta $%
. Note that this equilibrium only exists when the resulting distance is
strictly lower than $2\delta $, that is, when $2\kappa -2\delta <2\delta $ ($%
\frac{1}{2}<\frac{\delta }{\kappa }$).

When the ratio $\frac{\delta }{\kappa }$ is small enough so that this
condition is no longer satisfied, we move to a situation of full
differentiation.

\bigskip

\textbf{Case} $\frac{\delta }{\kappa }\leq \frac{1}{2}$\textbf{: Full
differentiation.} In that case, for each pair $(\delta ,\kappa )$, there is
a unique symmetric equilibrium $(-\delta ,\delta )$; besides, as soon as $%
\frac{\delta }{\kappa }<\frac{1}{2}$, there is also a continuum of
asymmetric Nash equilibria. In all these equilibria, the two firms are
located at distance $2\delta $ one from the other: the potential attraction
zones of the two firms do not intersect (more precisely, a mass zero of
consumers simultaneously belong to both potential attraction zones).

In the case of a normal distribution, the condition $\delta \leq \frac{1}{2}%
\kappa $ states that the domain of attraction of a firm located at the
center has to cover at most 45\% of the population.

\subsection{Efficiency of equilibria}

In this subsection, we compare equilibrium locations to these which would be
optimal either from the consumers' point of view (consumer surplus
maximizing locations) or from the firms' perspective (aggregate profit
maximizing locations). Aggregate profit maximizing locations and consumer
surplus maximizing locations are described in the following proposition.

\begin{proposition}
\label{prop:efficiency} ~~\textbf{(Efficiency)}\newline
$(1)$ \textbf{(Aggregate profit maximizing locations)} Assume that both
firsms have the same production functions ($\gamma _{1}=\gamma _{2}$). Then
the location profile maximizing the sum of the firms' profits is $(-\delta
,\delta )$, which entails full differentiation.\newline
$(2)$ \textbf{(Consumer surplus maximizing locations)} The location profile
maximizing consumer surplus entails partial differentiation. The detail of
the location profile maximizing consumer surplus depends on the
transportation cost function $c(.)$.
\end{proposition}

The proof of Proposition \ref{prop:efficiency} is provided in the appendix,
section \ref{Proof_efficiency}.

\bigskip

Figure 1 provides an illustration of Proposition \ref{prop:efficiency} when
consumers are distributed according to a standard normal distribution. It
plots the level of differentiation as a function of $\delta $: \textit{(i)}
in the Nash equilibria profile, \textit{(ii)} in the aggregate profit
maximizing profile, and \textit{(iii)} in the consumers' surplus maximizing
profile. More precisely, on the vertical axis, it shows the ratio of the
resulting distance between the two firms to the minimal distance between the
firms that guaranties full differentiation ($2\delta $). A ratio of $1$
means full differentiation and a ratio of $0$ means no differentiation. As
emphasized in Proposition \ref{prop:efficiency}, the profile of locations
that maximizes consumer surplus depends on the transportation cost function;
it this example, we choose linear transportation costs $c(d)=d$. As noted
earlier (see Example 4), when consumers are distributed according to $%
\mathcal{N}(1)$, $\kappa =\sqrt{2\ln (2)}$, which is approximately equal to $%
1.18$.

\begin{center}
\includegraphics[scale=0.3]{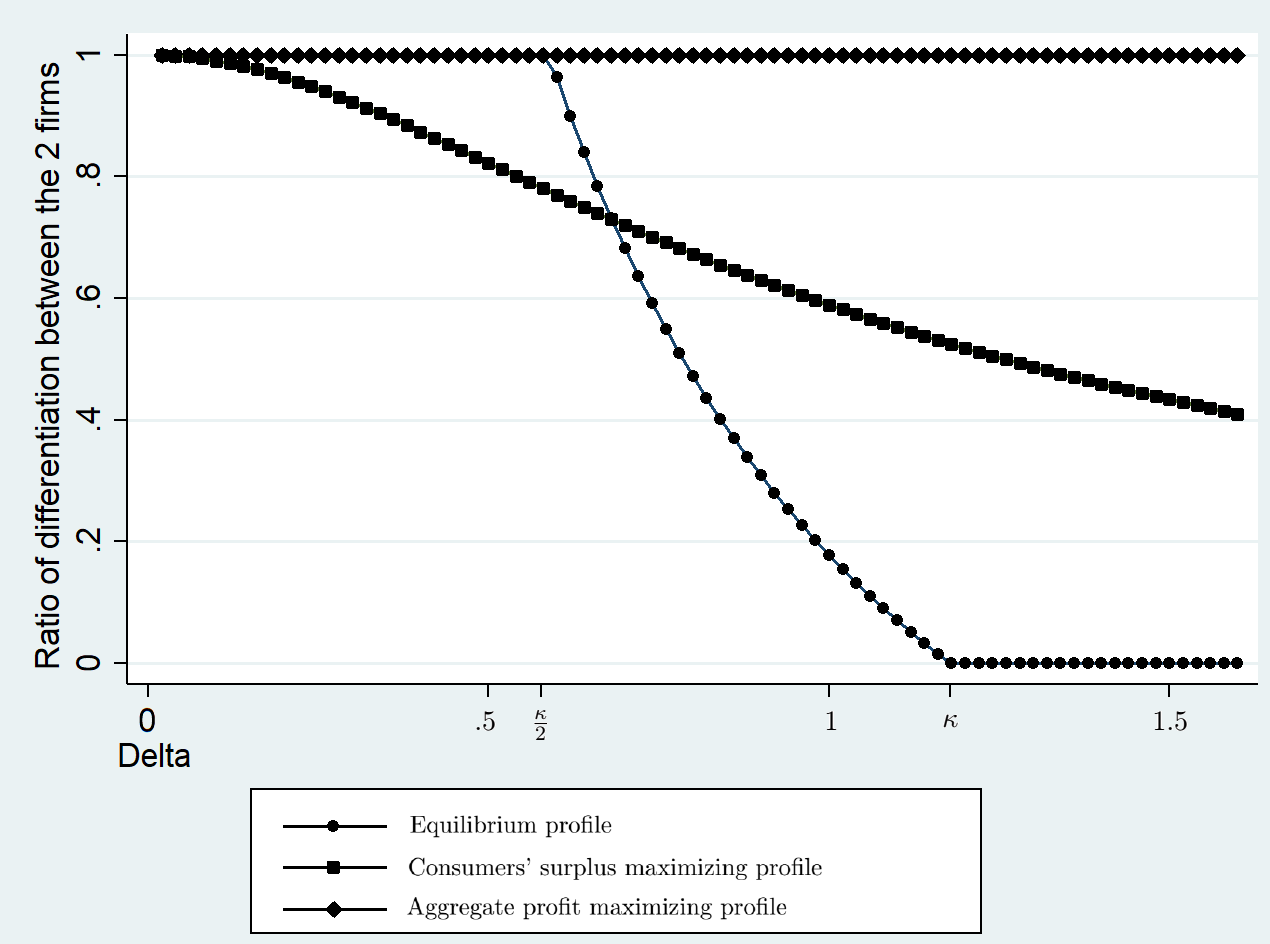}
\end{center}

This figure illustrates that there exists a unique value of $\delta \in
]0,+\infty \lbrack $ such that the equilibrium and the consumers surplus
maximizing profile coincide. For smaller value of $\delta $, the distance
between firms at equilibrium is strictly larger than it would be in a
consumer surplus maximizing profile, for larger value of $\delta $, it is
strictly smaller.

\subsection{Uniform distribution of consumers\label{section:uniform}}

So far, we have supposed that the distribution of consumers is symmetric
around $0$, log-concave and strictly decreasing on $\mathbb{R}^{+}$. These
assumptions are weak as this case includes most of standard distributions
(Normal, Laplace, Logistic, etc.). However, a large part of the literature
on horizontal differentiation has studied the particular case of consumers
uniformly distributed on an interval. In this subsection we discuss this
case, which is not included in our general model as in the uniform case, the
function $f$ is no longer strictly decreasing on $\mathbb{R}^{+}$.

We consider the case where consumers are distributed uniformly in the
interval $X=[-\kappa ,\kappa ]$, for some $\kappa >0$. We choose this
notation to be consistent with our previous notation (see Definition (\ref%
{Definition_kappa})). Indeed, consider the following alternative (more
general) definition for $\kappa $:%
\begin{equation*}
\kappa =\inf \left \{ t\in \mathbb{R}^{+}:\frac{1}{2}f(0)>f(t)\right \} .
\end{equation*}
It coincides with Definition (\ref{Definition_kappa}) when $f$ is continuous
and strictly decreasing on $\mathbb{R}^{+}$, but can also be used in the
uniform case. We still assume that the firms can choose any location on the
real line $\mathbb{R}$.

Proposition \ref{prop:unif} shows that most of the results stated in
Propositions 1, 2 and 3 extend to the uniform case, the only adaption to be
made being the characterization of equilibria with full differentiation.

\begin{proposition}
\label{prop:unif}~~\textbf{(Uniform distribution)}\newline
Assume that consumers are uniformly distributed on $[-\kappa ,\kappa ]$.%
\newline
Proposition 1, Proposition 2 and Part (i) of Proposition 3 extend to this
uniform case.\newline
The only difference lays with the characterization of equilibria with full
diferentiation (Part (ii) of Proposition 3): In the uniform case, if $\delta
\leq \frac{\kappa }{2}$, there exists a continuum of equilibria where the
two firms locate at distance at least $2\delta $ one from the other. More
specifically, supposing without loss of generality that $x_{1}\leq x_{2}$,
the whole set of equilibria is $\left( x_{1},x_{2}\right) $ for $-\kappa
+\delta \leq x_{1}\leq x_{1}+2\delta \leq x_{2}\leq \kappa -\delta .$
\end{proposition}

The proof of this proposition is provided in section \ref{Proof:Uniforrm}\
in the appendix. Proposition \ref{Proof:Uniforrm} shows that in the case of
a uniform distribution, the three regimes of no, partial, and full
differentiation still exist. The main difference is that now, in the case of
full differentiation, the firms can locate at a distance strictly larger
than $2\delta $ one from the other at equilibrium.

\section{Appendix\label{se:appendix}}

\subsection{Proof of Proposition \protect \ref{prop:no differentiation},
Proposition \protect \ref{prop:partial differentiation}, and Proposition 
\protect \ref{prop:full differentiation} (Characterization of equilibria) 
\label{se:proof}}

The proof of the propositions rely on the following lemmas.

\begin{lemma}
\label{le:PropertyF}~~\newline
If the density $f$ is symmetric around $0$ and strictly decreasing on $%
\mathbb{R}
_{+}$, then the cumulative function $F$ satisfies the following properties:%
\newline
For any $\delta >0$,\newline
(1) if $x<0$, then $F(x+\delta )-F(x)>F(x)-F(x-\delta )$\newline
(2) if $x>0$, then $F(x+\delta )-F(x)<F(x)-F(x-\delta )$.
\end{lemma}

\begin{proof}
\textbf{of Lemma \ref{le:PropertyF}}\newline
(1) Note that 
\begin{eqnarray*}
F(x+\delta )-F(x) &=&\int_{t=x}^{t=x+\delta
}f(t)dt=\int_{t=x}^{t=-x}f(t)dt+\int_{t=-x}^{t=x+\delta }f(t)dt \\
F(x)-F(x-\delta ) &=&\int_{t=x-\delta }^{t=x}f(t)dt=\int_{t=x-\delta
}^{t=-x-\delta }f(t)dt+\int_{t=-x-\delta }^{t=x}f(t)dt
\end{eqnarray*}%
Assume $x<0$.\newline
Consider first the case $x+\delta \leq -x$. Then for all $t\in \left[
x,x+\delta \right] $, $f(t)\geq f(x)$ with a strict inequality if $%
x<t<x+\delta $. Besides, for all $t\in \left[ x-\delta ,x\right] $, $%
f(t)\leq f(x)$ with a strict inequality if $t<x$. This shows that in that
case $F(x+\delta )-F(x)>F(x)-F(x-\delta )$.\newline
Consider now the case $x+\delta \geq -x$. By symmetry of $f$, $%
\int_{t=-x}^{t=x+\delta }f(t)dt=\int_{t=-x-\delta }^{t=x}f(t)dt$. For all $%
t\in \left[ x-\delta ,-x-\delta \right] $, $f(t) < f(x)$. And for all $t\in \left[ x,-x\right] $, $%
f(t)\geq f(x)$ with a strict inequality if $x<t<-x$. This shows that in that
case too $F(x+\delta )-F(x)>F(x)-F(x-\delta )$.\newline
(2) By symmetry, the proof is the same as for claim (1).
\end{proof}

Lemma \ref{le:4claims} provides a few useful remarks about the structure of
equilibria and best responses.

\begin{lemma}
\label{le:4claims}~~\newline
(1) $x_{2}=0$ is the unique best response to any $x_{1}$ such that $%
|x_{1}|\geq 2\delta $.\newline
(2) Any best response to $x_{1}\in \left[ -2\delta ,0\right[ $ belongs to
the interval $\left] x_{1},x_{1}+2\delta \right] $.\newline
(3) Any best response to $x_{1}\in \left] 0,2\delta \right] $ belongs to the
interval $\left[ x_{1}-2\delta ,x_{1}\right[ $.\newline
(4) Any best response to $x_{1}=0$ belongs to the interval $[-2\delta
,2\delta ]$.\newline
(5) Any equilibrium $(x_{1},x_{2})$ such that $x_{1}\leq x_{2}$ satisfies $%
x_{1}\in \lbrack -2\delta ,0]$ and $x_{2}\in \left[ 0,2\delta \right] $.
\end{lemma}

\begin{proof}
\textbf{of Lemma \ref{le:4claims}}\newline
(1) If $|x_{1}|\geq 2\delta $ then $q_{2}(x_{1},0)=F(\delta )-F(-\delta )$,
which is the strictly maximal feasible payoff since $f$ is symmetric around $%
0$ and strictly decreasing on $\mathbb{R}^{+}$.\newline
(2) Let $x_{1}\in \left[ -2\delta ,0\right[ $.\newline
First note that because $f$ is strictly decreasing on $\mathbb{R}^{+}$, if $%
x_{2}>x_{1}+2\delta $, then $q_{2}(x_{1},x_{2})<q_{2}(x_{1},x_{1}+2\delta )$%
, and if $x_{2}<x_{1}-2\delta $ then $q_{2}(x_{1},x_{2})<q_{2}(x_{1},x_{1}-2%
\delta )$. Therefore the best response belongs to the interval $%
[x_{1}-2\delta ,x_{1}+2\delta ]$.\newline
But, because $f$ is symmetric, if $x_{2}\in \lbrack x_{1}-2\delta ,x_{1}[$
then $q_{2}(x_{1},x_{2})<q_{2}(x_{1},-x_{2})$, which shows that the best
response belongs to the interval $[x_{1},x_{1}+2\delta ]$.\newline
It remains to show that $x_{2}=x_{1}$ cannot be a best response against $%
x_{1}\in \left[ -2\delta ,0\right[ $. Note that for $\varepsilon >0$ small
enough,%
\begin{eqnarray*}
q_{2}(x_{1},x_{1}) &=&\frac{F(x_{1}+\delta )-F(x_{1}-\delta )}{2} \\
q_{2}(x_{1},x_{1}+\varepsilon ) &=&F(x_{1}+\varepsilon +\delta )-F(x_{1}+%
\frac{\varepsilon }{2})
\end{eqnarray*}%
Therefore%
\begin{equation*}
\lim_{\substack{ \varepsilon \rightarrow 0  \\ \varepsilon >0}}%
q_{2}(x_{1},x_{1}+\varepsilon )-q_{2}(x_{1},x_{1})=\frac{F(x_{1}+\delta
)+F(x_{1}-\delta )}{2}-F(x_{1}),
\end{equation*}%
which by Lemma \ref{le:PropertyF} is positive since by assumption $x_{1}<0.$
This concludes the proof of claim (2).\newline
(3) By symmetry, the proof is the same as for claim (2).\newline
(4) Let $x_{1}=0$. If $x_{2}>2\delta $, then $q_{2}(0,x_{2})=F(x_{2}+\delta
)-F(x_{2}-\delta )$, which is strictly decreasing in $x_{2}$ for $x_{2}\in %
\left] 2\delta ,+\infty \right[ $. Therefore, $x_{2}>2\delta $ cannot be a
best response against $x_{1}=0$. By symmetry, $x_{2}<-2\delta $ cannot be a
best response against $x_{1}=0$.\newline
(5) Let $(x_{1},x_{2})$ be an equilibrium with $x_{1}\leq x_{2}$.\newline
Suppose first that $\left \vert x_{1}\right \vert >2\delta $. Then, $x_{2}=0$
according to claim (1), and claim (4) contradicts the fact that $x_{1}$ is a
best response to $x_{2}$, so it must be the case that $\left \vert
x_{1}\right \vert \leq 2\delta $.\newline
Suppose now that $0<x_{1}\leq 2\delta $. Then, by claim (3), $x_{2}<x_{1}$.
Since by assumption, $x_{1}\leq x_{2}$, it implies a contradiction.
Therefore $x_{1}\in \left[ -2\delta ,0\right] $.\newline
Similar arguments show that $x_{2}\in \left[ 0,2\delta \right] $, which
concludes the proof of claim (5).
\end{proof}

\bigskip

According to Lemma \ref{le:4claims}, at any equilibrium $(x_{1},x_{2})$ such
that $x_{1}\leq x_{2}$, we have that $x_{1}\in \lbrack -2\delta ,0]$ (claim
(5)) and Player 2's best response against $x_{1}$ belongs to the interval $%
\left] x_{1},x_{1}+2\delta \right] $ (claim (2)). Player 2's payoff when it
selects $x_{2}$ in this interval is:%
\begin{equation*}
q_{2}(x_{1},x_{2})=F(x_{2}+\delta )-F\left( \frac{x_{2}+x_{1}}{2}\right) 
\text{,}
\end{equation*}%
and 
\begin{equation}
\frac{\partial q_{2}}{\partial x_{2}}(x_{1},x_{2})=f(x_{2}+\delta )-\frac{1}{%
2}f\left( \frac{x_{2}+x_{1}}{2}\right).
\label{Derivative_q2}
\end{equation}
Note that when $x_{2}=x_{1}+2\delta $, we only compute a left derivative.

\begin{remark}
Before turning to the proofs of the propositions, let us introduce the
function $\Psi _{z}$, where for $t,z\in \mathbb{R}$, $\Psi _{z} $\ is
defined as follows: 
\begin{equation}
\Psi _{z}(t):=\frac{f(t)}{f(t+z)}.  \label{definition_psi}
\end{equation}%
Because of the log-concavity of $f$, we have that: 

\begin{itemize}
\item For any $z>0$, $t\longmapsto \Psi _{z}(t)$ is increasing in $t\in 
\mathbb{R}$,

\item For any $z<0$, $t\longmapsto \Psi _{z}(t)$ is decreasing in $t\in 
\mathbb{R}$.
\end{itemize}

Indeed, the log-concave function $f$ can be written $e^{g}$, where $g$ is a
concave function. Therefore $\Psi _{z}^{\prime }(t)=(g^{\prime
}(t)-g^{\prime}(t+z))e^{g(t)-g(t+z)}$ has the same sign than $g^{\prime
}(t)-g^{\prime }(t+z)$. The monotony of $\Psi_z$ follows from the fact that $%
g^{\prime }$ is decreasing.
\end{remark}

We are now ready to complete the proofs of propositions \ref{prop:no
differentiation}, \ref{prop:partial differentiation} and \ref{prop:full
differentiation}.

\subparagraph{\protect \bigskip}

\textbf{Proof of Propostion \ref{prop:no differentiation}}

Assume that $(x_{1},x_{2})$ is an equilibrium such that $x_{1}=x_{2}$. Lemma %
\ref{le:4claims} implies that $x_{1}=x_{2}=0$. Indeed, Claim (5) states that 
$x_{1}\in \lbrack -2\delta ,0]$, and Claim (2) states that any best response
to $x_{1}\in \left[ -2\delta ,0\right[ $ is strictly larger than $x_{1}$.
Therefore, there exists an equilibrium with no differentiation if and only
if $(0,0)$ is a Nash equilibrium. In that case, it is the unique equilibrium
with no differentiation.

It remains to show $(0,0)$ is a Nash equilibrium if and only if $\kappa \leq
\delta $. A necessary condition for $(0,0)$ to be a Nash equilibrium is that 
\begin{equation*}
\lim_{\substack{ x_{2}\rightarrow 0  \\ x_{2}>0}}\frac{\partial q_{2}}{%
\partial x_{2}}(0,x_{2})\leq 0,
\end{equation*}%
which given equation (\ref{Derivative_q2}) can be written as $f(\delta )\leq \frac{1}{%
2}f(0)$. This is exactly condition $\kappa \leq \delta $. Although the
function $x_2 \mapsto g_2(x_1,x_2)$ is in general discontinuous in $x_2=x_1$%
, it is continuous in the particular case where $x_1=0$ (because $f$ is
symmetric), so that it's enough to consider the derivative in $x_2>0$, $x_2
\rightarrow 0$.

Last, let us show that when condition $\kappa \leq \delta $ holds, $\frac{%
\partial q_{2}}{\partial x_{2}}(0,x_{2})<0$ for all $x_{2}\in \left]
0,2\delta \right] $, which will guarantee that $x_{2}=0$ is a best response
againt $x_{1}=0$. Given equation (\ref{Derivative_q2}), for $x_{2}\in \left]
0,2\delta \right] $: 
\begin{equation*}
\frac{\partial q_{2}}{\partial x_{2}}(0,x_{2})=f(x_{2}+\delta )-\frac{1}{2}f(%
\frac{x_{2}}{2}).
\end{equation*}
Note that when $x_{2}=2\delta $, we only compute a left derivative. We
have: 
\begin{equation*}
\frac{f(\frac{x_{2}}{2})}{f(x_{2}+\delta )}=\Psi _{\frac{x_{2}}{2}+\delta
}\left( \frac{x_{2}}{2}\right) \geq \Psi _{\frac{x_{2}}{2}+\delta }\left(
0\right) =\frac{f(0)}{f(\frac{x_{2}}{2}+\delta )}>\frac{f(0)}{f(\delta )},
\end{equation*}%
where the first inequality follows from the observation that $\frac{x_{2}}{2}%
+\delta >0$ and $\frac{x_{2}}{2}>0$, and the second inequality follows from
the fact that $0<\delta <\frac{x_{2}}{2}+\delta $ and $f$ is strictly
decreasing on $\mathbb{R}^{+}$. Since by assumption $f(\delta )\leq \frac{1}{%
2}f(0)$, this proves that $\frac{\partial q_{2}}{\partial x_{2}}(0,x_{2})<0$
for all $x_{2}\in \left] 0,2\delta \right] $,and $x_{2}=0$ is a best
response againt $x_{1}=0$.

The same argument shows that when $f(\delta )\leq \frac{1}{2}f(0)$, $x_{1}=0$
is a best response againt $x_{2}=0$.

Therefore, condition $\kappa \leq \delta $ is a necessary and sufficient
condition for $(0,0)$ to be a Nash equilibrium. This concludes the proof of
Proposition \ref{prop:no differentiation}.

\bigskip

\subparagraph{Proof of Proposition \protect \ref{prop:partial differentiation}%
.}

Assume that $x_1 \ leq x_2$ and that $(x_{1},x_{2})$ is an equilibrium
with partial differentiation, meaning that $0<x_{2}-x_{1}<2\delta $.

According to Lemma \ref{le:4claims}, at any equilibrium $(x_{1},x_{2})$ such
that $x_{1}\leq x_{2}$, it must be the case that $x_{1}\in \lbrack -2\delta
,0]\cap \left[ x_{2}-2\delta ,x_{2}\right[ $ and $x_{2}\in \lbrack -2\delta
,0]\cap \left] x_{1},x_{1}+2\delta \right] $. Since $x_{2}-2\delta
<x_{1}<x_{2}<x_{1}+2\delta $, the first-order conditions imply that $\frac{%
\partial q_{1}}{\partial x_{1}}\left( x_{1},x_{2}\right) =\frac{\partial
q_{2}}{\partial x_{2}}\left( x_{1},x_{2}\right) =0$, and therefore that: 
\begin{equation*}
f(x_{2}+\delta )=f(x_{1}-\delta ),
\end{equation*}%
which, when $x_{2}-x_{1}<2\delta $, is possible only if $x_{2}+x_{1}=0$: An
equilibrium with partial differentiation is necessarily symmetric.

Assume that $(-x_{2},x_{2})$ is a symmetric equilibrium with $x_{2}>0$. From
equation (\ref{Derivative_q2}), it must be the case that 
\begin{equation*}
\frac{\partial q_{2}}{\partial x_{2}}(-x_{2},x_{2})=0\Leftrightarrow
f(x_{2}+\delta )=\frac{1}{2}f\left( 0\right) .
\end{equation*}%
If $(-x_{2},x_{2})$ is an equilibrium partial differentiation, it must be
the case that $x_{2}<\delta $. By definition of $\kappa $ (see (\ref%
{Definition_kappa})), equation $f(x_{2}+\delta )=\frac{1}{2}f\left( 0\right) 
$ has a solution in $\left] 0,\delta \right[ $ if and only if $\frac{\kappa 
}{2}<\delta <\kappa $. In that case, the solution is unique and is $%
x_{2}=\kappa -\delta $.

Assume that $\frac{\kappa }{2}<\delta <\kappa $. It remains to show $(\delta
-\kappa ,\kappa -\delta )$ is a Nash equilibrium.

Let us first show that $x_{2}=\kappa -\delta $ is a best response against $%
x_{1}=\delta -\kappa $. It is sufficient to prove that:%
\begin{eqnarray*}
\frac{\partial q_{2}}{\partial x_{2}}(\delta -\kappa ,x_{2}) &>&0\text{ if }%
\delta -\kappa <x_{2}<\kappa -\delta \\
\frac{\partial q_{2}}{\partial x_{2}}(\delta -\kappa ,x_{2}) &<&0\text{ if }%
\kappa -\delta <x_{2}<3\delta -\kappa
\end{eqnarray*}

Given (\ref{Derivative_q2}), for $x_{2}\in \left] \delta -\kappa ,3\delta
-\kappa \right] $: 
\begin{equation*}
\frac{\partial q_{2}}{\partial x_{2}}(\delta -\kappa ,x_{2})=f(x_{2}+\delta
)-\frac{1}{2}f\left( \frac{\delta -\kappa +x_{2}}{2}\right).
\end{equation*}


Consider first the case $\delta -\kappa <x_{2}<\kappa -\delta $. Note that%
\begin{equation*}
\frac{f(\frac{\delta -\kappa +x_{2}}{2})}{f\left( x_{2}+\delta \right) }%
=\Psi _{\frac{\delta +\kappa +x_{2}}{2}}\left( \frac{\delta -\kappa +x_{2}}{2%
}\right) \leq \Psi _{\frac{\delta +\kappa +x_{2}}{2}}\left( 0\right) =\frac{%
f(0)}{f(\frac{\delta +\kappa +x_{2}}{2})}<\frac{f(0)}{f(\kappa )},
\end{equation*}%
where the first inequality follows from the observation that $\frac{\delta
+\kappa +x_{2}}{2}>0$ and $\frac{\delta -\kappa +x_{2}}{2}<0$, and the
second inequality follows from the fact that $0<\frac{\delta +\kappa +x_{2}}{%
2}<\kappa $ and $f$ is strictly decreasing on $%
\mathbb{R}
^{+}$. Since by assumption $f(\kappa )=\frac{1}{2}f(0)$, this proves that $%
\frac{\partial q_{2}}{\partial x_{2}}(0,x_{2})>0$ for all $x_{2}$ such that $%
\delta -\kappa <x_{2}<\kappa -\delta $.

Consider now the case $\kappa -\delta <x_{2}<3\delta -\kappa $. Note that%
\begin{equation*}
\frac{f(\frac{\delta -\kappa +x_{2}}{2})}{f\left( x_{2}+\delta \right) }%
=\Psi _{\frac{\delta +\kappa +x_{2}}{2}}\left( \frac{\delta -\kappa +x_{2}}{2%
}\right) \geq \Psi _{\frac{\delta +\kappa +x_{2}}{2}}\left( 0\right) =\frac{%
f(0)}{f(\frac{\delta +\kappa +x_{2}}{2})}>\frac{f(0)}{f(\kappa )},
\end{equation*}%
where the first inequality follows from the observation that $\frac{\delta
+\kappa +x_{2}}{2}>0$ and $\frac{\delta -\kappa +x_{2}}{2}>0$, and the
second inequality follows from the fact that $0<\kappa <\frac{\delta +\kappa
+x_{2}}{2}$ and $f$ is strictly decreasing on $%
\mathbb{R}
^{+}$. Since by assumption $f(\kappa )=\frac{1}{2}f(0)$, this proves that $%
\frac{\partial q_{2}}{\partial x_{2}}(0,x_{2})<0$ for all $x_{2}$ such that $%
\kappa -\delta <x_{2}<3\delta -\kappa $.

This shows that $x_{2}=\kappa -\delta $ is a best response against $%
x_{1}=\delta -\kappa $.

A symmetric argument shows that $x_{1}=\delta -\kappa $ is a best response
against $x_{2}=\kappa -\delta $

Proposition \ref{prop:partial differentiation} is proved.

\bigskip

\subparagraph{Proof of proposition \protect \ref{prop:full differentiation}.}

It follows from Lemma \ref{le:4claims} that any equilibrium with full
differentiation is necessarily of the form $(a-\delta ,a+\delta )$ for some $%
a\in \lbrack -\delta ,\delta ]$. Besides, $(a-\delta ,a+\delta )$ is an
equilibrium only if: 
\begin{eqnarray*}
\lim_{\substack{ x_{1}\rightarrow a-\delta  \\ x_{1}>a-\delta }}\frac{%
\partial q_{1}}{\partial x_{1}}\left( x_{1},a+\delta \right) &=&\frac{1}{2}%
f\left( a\right) -f\left( a-2\delta \right) \leq 0\Leftrightarrow \Psi
_{-2\delta }(a)\leq 2, \\
\lim_{\substack{ x_{2}\rightarrow a+\delta  \\ x_{2}<a+\delta }}\frac{%
\partial q_{2}}{\partial x_{2}}\left( a-\delta ,x_{2}\right) &=&f\left(
a+2\delta \right) -\frac{1}{2}f\left( a\right) \geq 0\Leftrightarrow \Psi
_{2\delta }(a)\leq 2,
\end{eqnarray*}%
where $\Psi _{z}(.)$ is defined by (\ref{definition_psi}). Therefore, a
necessary condition for $(a-\delta ,a+\delta )$ to be an equilibrium is that 
$\max \left( \Psi _{-2\delta }(a),\Psi _{2\delta }(a)\right) \leq 2$.

Note that because of the logconcavity of $f$, $t\mapsto \Psi _{-2\delta }(t)$
is a decreasing function on $%
\mathbb{R}
$ and $t\mapsto \Psi _{2\delta }(t)$\ is an increasing function on $%
\mathbb{R}
$. Besides, $\Psi _{-2\delta }(0)=\Psi _{2\delta }(0).$ Therefore 
\begin{equation*}
\max \left( \Psi _{-2\delta }(a),\Psi _{2\delta }(a)\right) =\left \{ 
\begin{array}{c}
\Psi _{2\delta }(a)\text{ if }a\geq 0 \\ 
\Psi _{-2\delta }(a)\text{ if }a\leq 0%
\end{array}%
\right.
\end{equation*}%
Since $\Psi _{-2\delta }(a)=\Psi _{2\delta }(-a)$, then $\max \left( \Psi
_{-2\delta }(a),\Psi _{2\delta }(a)\right) =\Psi _{2\delta }(\left \vert
a\right \vert ).$

Note that there exists $a\in \lbrack -\delta ,\delta ]$ such that $\Psi
_{2\delta }(\left \vert a\right \vert )\leq 2$ if and only if $\Psi
_{2\delta }(0)\leq 2\Leftrightarrow 2\delta \leq \kappa $.

If this condition holds, $\alpha (f,\delta )$ defined in \ref%
{Definition_alpha} exists and is uniquely defined, and $(-\delta +a,+\delta
+a)$ is an equilibrium only if $a\in \lbrack -\alpha (f,\delta ),\alpha
(f,\delta )]$.

Note that the case $a=0$ is the unique symmetric equilibria in this class.

Assume that $2\delta \leq \kappa $ and consider $x_{1}\in \lbrack -\delta
-\alpha (f,\delta ),-\delta +\alpha (f,\delta )]$. Let us show that $\frac{%
\partial q_{2}}{\partial x_{2}}(x_{1},x_{2})>0$ for all $x_{2}\in \left]
x_{1},x_{1}+2\delta \right] $, which will guarantee that $%
x_{2}=x_{1}+2\delta $ is a best response againt $x_{1}$. Given (\ref%
{Derivative_q2}), for $x_{2}\in \left] x_{1},x_{1}+2\delta \right[ $: 
\begin{equation*}
\frac{\partial q_{2}}{\partial x_{2}}(x_{1},x_{2})=f(x_{2}+\delta )-\frac{1}{%
2}f(\frac{x_{1}+x_{2}}{2}).\text{\footnote{%
Note that when $x_{2}=x_{1}+2\delta $, we only compute a left derivative.}}
\end{equation*}%
Note that%
\begin{equation*}
\frac{f(\frac{x_{1}+x_{2}}{2})}{f(x_{2}+\delta )}=\Psi _{\delta +\frac{%
x_{2}-x_{1}}{2}}\left( \frac{x_{1}+x_{2}}{2}\right) \leq \Psi _{\delta +%
\frac{x_{2}-x_{1}}{2}}\left( x_{1}+\delta \right) =\frac{f(x_{1}+\delta )}{f(%
\frac{x_{1}+x_{2}}{2}+2\delta )}<\frac{f(x_{1}+\delta )}{f(x_{1}+\delta
+2\delta )},
\end{equation*}%
where the first inequality follows from the observation that $\delta +\frac{%
x_{2}-x_{1}}{2}>0$ and $\frac{x_{1}+x_{2}}{2}<x_{1}+\delta $, the second
inequality follows from the fact that $0\leq \frac{x_{1}+x_{2}}{2}+2\delta
<\left( x_{1}+\delta \right) +2\delta $ and $f$ is strictly decreasing on $%
\mathbb{R}
^{+}$.

Note also that%
\begin{equation*}
\frac{f(x_{1}+\delta )}{f(x_{1}+\delta +2\delta )}=\Psi _{2\delta }\left(
x_{1}+\delta \right) \leq \Psi _{2\delta }\left( \alpha (f,\delta )\right) =%
\frac{f(\alpha (f,\delta ))}{f(\alpha (f,\delta )+2\delta )}\leq 2
\end{equation*}%
where the first inequality follows from the observation that $2\delta >0$
and $x_{1}+\delta \leq \alpha (f,\delta )$, and the second inequality
follows from the definition of $\alpha (f,\delta )$ (see (\ref%
{Definition_alpha})).

This proves that $\frac{\partial q_{2}}{\partial x_{2}}(x_{1},x_{2})<0$ for
all $x_{2}\in \left] x_{1},x_{1}+2\delta \right] $,and $x_{2}=x_{1}+2\delta $
is a best response againt $x_{1}$.

This concludes the proof of Proposition \ref{prop:full differentiation}.

\subsection{Proof of Proposition \protect \ref{prop:efficiency} (Efficiency) 
\label{Proof_efficiency}}

\textbf{Claim (1): Consumers' surplus.}

Suppose that firms' locations are $x_{1}$ and $x_{2}$. When $x_{1}\leq x_{2}$%
, the consumers' surplus is:%
\begin{eqnarray*}
&&CS(x_{1},x_{2}) \\
&:&=\left \{ 
\begin{array}{l}
\displaystyle \int_{x_{1}-\delta }^{\frac{x_{1}+x_{2}}{2}%
}(v-p-c(|x_{1}-t|))f(t)dt+\int_{\frac{x_{1}+x_{2}}{2}}^{x_{2}+\delta
}(v-p-c(|x_{2}-t|))f(t)dt\text{ if }|x_{2}-x_{1}|\leq 2\delta , \\ 
~~ \\ 
\displaystyle \int_{x_{1}-\delta }^{x_{1}+\delta
}(v-p-c(|x_{1}-t|))f(t)dt+\int_{x_{2}-\delta }^{x_{2}+\delta
}(v-p-c(|x_{2}-t|))f(t)dt\text{ if }|x_{2}-x_{1}|\geq 2\delta .%
\end{array}%
\right.
\end{eqnarray*}

Because players are anonymous, we have $CS(x_1,x_2)=CS(x_2,x_1)$. Therefore,
the previous expression also holds for $x_{1}\geq x_{2}$.

Note first that there exists a profile that maximizes the consumers'
surplus. Indeed, the surplus is a continuous function and because $f(x)$ goes to zero
as $|x|$ goes to $+\infty$, we can restrict the analysis of $CS(x_1,x_2)$ on
a compact subset of $\mathbb{R} \times \mathbb{R}$.

Remark that a situation with full convergence cannot be optimum. Indeed, for
any $x,y\in \mathbb{R}$, $x\neq y$, $CS(x,y)>CS(x,x)$.\newline

It remains to show that a situation where $|x_{2}-x_{1}|\geq 2\delta $
cannot be an optimum. Straightforward computations show that when firms'
locations are $x_{1}$ and $x_{2}$, with $x_{1}\leq x_{2}$ and $%
|x_{2}-x_{1}|\geq 2\delta $, then:\textit{\ }%
\begin{eqnarray*}
\lim_{\substack{ \varepsilon \rightarrow 0  \\ \varepsilon >0}}\frac{%
CS(x_{1}+\varepsilon ,x_{2})-CS(x_{1},x_{2})}{\varepsilon }
&=&\int_{0}^{\delta }c^{\prime }(s)\left[ f(x_{1}+s)-f(x_{1}-s)\right] ds, \\
\lim_{\substack{ \varepsilon \rightarrow 0  \\ \varepsilon >0}}\frac{%
CS(x_{1},x_{2}-\varepsilon )-CS(x_{1},x_{2})}{\varepsilon }
&=&\int_{0}^{\delta }c^{\prime }(s)\left[ f(x_{2}-s)-f(x_{2}+s)\right] ds.
\end{eqnarray*}

Note that if $x_{1}<0$, for any $s$ such that $0\leq s\leq \delta $, $%
f(x_{1}+s)-f(x_{1}-s)>0$. Indeed, if $x_{1}+s\leq 0$, it is true since $f$
is increasing on $\mathbb{R}_{-}$. And if $x_{1}+s\geq 0$: $%
f(x_{1}+s)-f(x_{1}-s)>0\Leftrightarrow x_{1}+s<\left \vert
x_{1}-s\right
\vert \Leftrightarrow x_{1}<0$. Therefore if $x_{1}<0$, then $%
f(x_{1}+s)-f(x_{1}-s)>0$ and $\lim_{\substack{ \varepsilon \rightarrow 0  \\ %
\varepsilon >0}}\frac{CS(x_{1}+\varepsilon ,x_{2})-CS(x_{1},x_{2})}{%
\varepsilon }>0$: the consumer surplus would increase if Firm 1 were to move
closer to Firm 2.

Consider now the case $x_{1}\geq 0$. Since $|x_{2}-x_{1}|\geq 2\delta $, it
implies that $x_{2}>0$. A similar argument shows that if $x_{2}>0$, for any $%
s$ such that $0\leq s\leq \delta $, $f(x_{2}+s)-f(x_{2}-s)<0\ $and $\lim 
_{\substack{ \varepsilon \rightarrow 0  \\ \varepsilon >0}}\frac{%
CS(x_{1},x_{2}-\varepsilon )-CS(x_{1},x_{2})}{\varepsilon }>0$: the consumer
surplus would increase if Firm 2 were to move closer to Firm 1.

This completes the proof of part 1.

\bigskip

\textbf{Claim (2): Aggregate profit.}

Assume the following conditions hold: (i) $\gamma _{1}=\gamma _{2}=\gamma $;
(ii) $\gamma ^{\prime \prime }\geq 0$; (iii) $\gamma ^{\prime }(1)<p$. For $%
q\in \left[ 0,1\right] $ and $\alpha \in \left[ 0,1\right] $, denote by $\Pi
(q,\alpha )$ the aggregate profit when total production is $q$ and Firm 1
realizes a share $\alpha $ of the total production (the remaining share
being produce by Firm 2): 
\begin{equation*}
\Pi (q,\alpha )=pq-\gamma \left( \alpha q\right) -\gamma \left( (1-\alpha
)q\right) .
\end{equation*}

Then%
\begin{equation*}
\frac{\partial \Pi }{\partial q}(q,\alpha )=p-\alpha \gamma ^{\prime }\left(
\alpha q\right) -(1-\alpha )\gamma ^{\prime }\left( (1-\alpha )q\right) .
\end{equation*}%
Since $\gamma ^{\prime }(x)<p$ for all $x\in \left[ 0,1\right] $, $\frac{%
\partial \Pi }{\partial q}(q,\alpha )>0$: aggregate profit is increasing
with the aggregate output.

Note also that: 
\begin{eqnarray*}
\frac{\partial \Pi }{\partial \alpha }(q,\alpha ) &=&-q\gamma ^{\prime
}\left( \alpha q\right) +q\gamma ^{\prime }\left( (1-\alpha )q\right) \\
\frac{\partial ^{2}\Pi }{\partial \alpha ^{2}}(q,\alpha ) &=&-q^{2}\gamma
^{\prime \prime }\left( \alpha q\right) -q^{2}\gamma ^{\prime \prime }\left(
(1-\alpha )q\right) \leq 0,
\end{eqnarray*}%
therefore for all $\alpha \in \left[ 0,1\right] $,%
\begin{equation*}
\Pi (q,\alpha )\leq \Pi \left( q,\frac{1}{2}\right) ,
\end{equation*}%
which means that fixing the total output $q$, an equal sharing of the
production is efficient (there is no way to make costs strictly lower).

Since the profile of location $(-\delta ,+\delta )$ is the unique profile
which maximizes total sales, and since it is symmetric, it is the unique
solution of the aggregate profit maximization program. Which concludes the
proof of Part 2.

\subsection{\label{Proof:Uniforrm}Proof of Proposition \protect \ref%
{prop:unif} (Uniform case).}

Assume that consumers are uniformly distributed on the $\left[ -\kappa
,\kappa \right] $ interval. Firms can choose any location on the real line.

The proof straightforwardly follows from the following five lemmas.

\begin{lemma}
\label{le:support}~~\newline
\textit{If }$\left( x_{1},x_{2}\right) $\textit{\ is an equilibrium, then,
necessarily, }$x_{1},x_{2}\in \left[ -\kappa ,\kappa \right] .$
\end{lemma}

\begin{proof}
\textbf{of Lemma \ref{le:support}}\newline
Note first that at equilibrium, both firms receive a positive payoff.
Indeed, a firm could secure a positive payoff by moving to the center $0$.
This remark proves that at equilibrium, $x_{1},x_{2}\in \left] -\kappa
-\delta ,\kappa +\delta \right[ $.\newline
Assume that $\left( x_{1},x_{2}\right) $ is an equilibrium, with $x_{1}\leq
x_{2}$.\newline
Assume that $-\kappa -\delta \leq x_{1}<-\kappa $. If $x_{1}=x_{2}$, Firm 2
could strictly increase its payoff by moving to position $-x_{1}$, which
contradicts the fact that $\left( x_{1},x_{2}\right) $ is an equilibrium. If 
$x_{1}<x_{2}$, Firm 1 could strictly increase its payoff by moving slightly
closer to Firm 2, which again contradicts the fact that $\left(
x_{1},x_{2}\right) $ is an equilibrium.\newline
By symmetry, there can be no equilibria where $\kappa <x_{2}\leq \kappa
+\delta .$\newline
This concludes the proof of Lemma \ref{le:support}.
\end{proof}

\begin{lemma}
\label{le:unif_fulldiff}~~\newline
\textit{There exists an equilibrium where the firms locate at distance at
least }$2\delta $\textit{\ one from the other if and only if }$\delta \leq 
\frac{\kappa }{2}$\textit{. In that case, there exists a continuum of
equilibria. More specifically, supposing without loss of generality that }$%
x_{1}\leq x_{2}$\textit{, the whole set of equilibria is }$\left(
x_{1},x_{2}\right) $\textit{\ such that }$-\kappa +\delta \leq x_{1}\leq
x_{1}+2\delta \leq x_{2}\leq \kappa -\delta .$
\end{lemma}

\begin{proof}
\textbf{of Lemma \ref{le:unif_fulldiff}}\newline
By Lemma \ref{le:support}, one can restrict attention to $\left(
x_{1},x_{2}\right) $ such that $-\kappa \leq x_{1}\leq x_{1}+2\delta \leq
x_{2}\leq \kappa $. Note that it must be the case that $x_{2}\leq \kappa
-\delta $. Otherwise, Firm 2 could strictly increase its payoff by moving
slightly to the left. Similarly, it must be the case that $x_{1}\geq -\kappa
+\delta $. These two conditions, together with the fact that $x_{1}+2\delta
\leq x_{2}$ implies that $2\delta \leq \kappa $.\newline
Last, note that if $2\delta \leq \kappa $, any $\left( x_{1},x_{2}\right) $
such that $-\kappa +\delta \leq x_{1}\leq x+2\delta \leq x_{2}\leq \kappa
-\delta $ gives both firms the maximal possible payoff ($\frac{\delta }{%
\kappa }$), and is thus an equilibrium.
\end{proof}

\begin{lemma}
\label{le:unif_partialdiff}~~\newline
\textit{There exists an equilibrium where the firms locate at distance less
than }$2\delta $\textit{\ one from the other without converging if and only
if }$\frac{1}{2}<\frac{\delta }{\kappa }<1$\textit{. If this condition
holds, the unique equilibrium is }$\left( \delta -\kappa ,\kappa -\delta
\right) $\textit{\ (up to a permutation of the players).}
\end{lemma}

\begin{proof}
\textbf{of Lemma \ref{le:unif_partialdiff}}\newline
By Lemma \ref{le:support}, one can restrict attention to $\left(
x_{1},x_{2}\right) $ such that $-\kappa \leq x_{1}<x_{2}<x_{1}+2\delta $ and 
$x_{2}\leq \kappa $. In that case, $q_{2}\left( x_{1},x_{2}\right) =\frac{1}{%
2\kappa }\left[ \min (\kappa ,x_{2}+\delta )-\frac{x_{1}+x_{2}}{2}\right] $.
The fact that firm 2 cannot increase its payoff by deviating slightly to the
left implies that $x_{2}+\delta \geq \kappa $. The fact that firm 2 cannot
increase its payoff by deviating slightly to the right implies that $%
x_{2}+\delta \leq \kappa $. It must therefore be the case that $x_{2}=\kappa
-\delta $. Similarly, it must be the case that $x_{1}=-\kappa +\delta $. The
distance between firm 1 and 2 has to be $2\kappa -2\delta $. Since we have
imposed that this distance should be positive and less than $2\delta $, on
gets the following necessary condition: $\frac{1}{2}<\frac{\delta }{\kappa }%
<1$\textit{.}\newline
One may easily check that if this condition holds, $\left( \delta -\kappa
,\kappa -\delta \right) $ is an equilibrium.
\end{proof}

\begin{lemma}
\label{le:unif_00}~~\newline
$\left( 0,0\right) $\textit{\ is an equilibrium if and only if }$\delta \geq
\kappa $\textit{.}
\end{lemma}

\begin{proof}
\textbf{of Lemma \ref{le:unif_00}}\newline
Let us first show that if $\delta \geq \kappa $, $\left( 0,0\right) $\textit{%
\ }is an equilibrium. Assume that $\delta \geq \kappa $. Note that $%
q_{1}\left( 0,0\right) =q_{2}\left( 0,0\right) =\frac{1}{2}$. And note that 
\begin{eqnarray*}
q_{2}\left( 0,x_{2}\right) &=&\frac{1}{2\kappa }\left[ \kappa -\frac{x_{2}}{2%
}\right] =\frac{1}{2}-\frac{x_{2}}{4\kappa }\text{ if }x_{2}\leq 2\kappa \\
&=&0\text{ otherwise.}
\end{eqnarray*}%
This proves that $\forall x_{2}>0$, $q_{2}\left( 0,x_{2}\right) <q_{2}\left(
0,0\right) $. Therefore, $\left( 0,0\right) $\textit{\ }is an equilibrium.%
\newline
Let us now complete the proof by showing that is \textit{if }$\delta <\kappa 
$\textit{, }$\left( 0,0\right) $\textit{\ }is not an equilibrium. If $\delta
<\kappa $%
\begin{equation*}
q_{2}\left( 0,0\right) =\frac{\delta }{2 \kappa }
\end{equation*}%
and for $0<\varepsilon <\kappa -\delta $, 
\begin{equation*}
q_{2}\left( 0,\varepsilon \right) =\frac{1}{2\kappa }\left[ \varepsilon
+\delta -\frac{\varepsilon }{2}\right] =\frac{\delta }{2 \kappa }+\frac{%
\varepsilon }{4\kappa }>q_{2}\left( 0,0\right) \text{.}
\end{equation*}%
This completes the proof.
\end{proof}

\begin{lemma}
\label{le:unif_nodiff}~~\newline
\textit{There is no equilibrium where the firms choose the same location if
this location is different from }$0.$
\end{lemma}

\begin{proof}
\textbf{of Lemma \ref{le:unif_nodiff}}\newline
Assume that $\left( x_{1},x_{2}\right) $ is an equilibrium such that $%
x_{1}=x_{2}=x>0.$\newline
By Lemma \ref{le:support}, we know that necessarily, $0<x\leq \kappa $.%
\newline
Then: 
\begin{equation*}
q_{1}\left( x,x\right) =\frac{1}{2}\times \frac{1}{2\kappa }\left[ \min
\left( x+\delta ,\kappa \right) -\max \left( x-\delta ,-\kappa \right) %
\right]
\end{equation*}%
and for $\varepsilon >0$ small enough:%
\begin{equation*}
q_{1}\left( x-\varepsilon ,x\right) =\frac{1}{2\kappa }\left[ x-\frac{%
\varepsilon }{2}-\max \left( x-\varepsilon -\delta ,-\kappa \right) \right] .
\end{equation*}%
For $\left( x,x\right) $ to be an equilibrium, it is therefore necessary
that 
\begin{equation*}
x-\max \left( x-\delta ,-\kappa \right) \leq \frac{1}{2}\times \left[ \min
\left( x+\delta ,\kappa \right) -x+x-\max \left( x-\delta ,-\kappa \right) %
\right] ,
\end{equation*}%
which is equivalent to%
\begin{eqnarray*}
x-\max \left( x-\delta ,-\kappa \right) &\leq &\min \left( x+\delta ,\kappa
\right) -x \\
x+\min \left( -x+\delta ,\kappa \right) &\leq &\min \left( x+\delta ,\kappa
\right) -x \\
\min \left( \delta ,x+\kappa \right) &\leq &\min \left( \delta ,\kappa
-x\right)
\end{eqnarray*}%
If $\min \left( \delta ,x+\kappa \right) =x+\kappa $, this implies $x+\kappa
\leq \kappa -x$, which is impossible since by assumption $x>0$.\newline
Therefore $\min \left( \delta ,x+\kappa \right) =\delta $ and $x\leq \kappa
-\delta .$ Note that since $x>0$, this implies that $\delta <\kappa .$%
\newline
If $x\leq \kappa -\delta ,$ then $\min \left( x+\delta ,\kappa \right)
=x+\delta $ and $\max \left( x-\delta ,-\kappa \right) =x-\delta $,
therefore: 
\begin{equation*}
q_{1}\left( x,x\right) =\frac{1}{2}\times \frac{\delta }{\kappa }.
\end{equation*}%
But then note that 
\begin{eqnarray*}
q_{1}\left( -x,x\right) &=&\frac{1}{2\kappa }\left[ \min \left( -x+\delta
,0\right) -(-x+\delta )\right] \\
&=&\frac{1}{2}\times \frac{\delta }{\kappa }+\frac{1}{2\kappa }x+\min \left(
-x+\delta ,0\right) .
\end{eqnarray*}%
Since $x>0$, $q_{1}\left( -x,x\right) >q_{1}\left( x,x\right) $ and $\left(
x,x\right) $ is not an equilibrium.
\end{proof}

\bibliographystyle{plainnat}
\bibliography{bib_unitdemand}

\end{document}